\documentclass[12pt,a4paper]{article}
\usepackage{amsmath,amssymb,amscd,amsthm,color,amsrefs}

\begin{document}

% ********************

% Theorem

\newtheorem{theo}{ {\bf Theorem} }
\newtheorem{lemm}[theo]{ {\bf Lemma} }
\newtheorem{prop}[theo]{ {\bf Proposition} }
\newtheorem{cor}[theo]{ {\bf Corollary} }
\newtheorem{ques}[theo]{ {\bf Question} }
\newtheorem{rema}[theo]{ {\bf Remark} }
\newtheorem{defini}[theo]{ {\bf Definition} }

% Numbers

\newcommand{\NN}{{\mathbb N}}
\newcommand{\ZZ}{{\mathbb Z}}
\newcommand{\QQ}{{\mathbb Q}}
\newcommand{\RR}{{\mathbb R}}
\newcommand{\CC}{{\mathbb C}}
\newcommand{\HH}{{\mathbb H}}
\newcommand{\HHH}{{\cal H}}

% Vectors

\newcommand{\bs}[1]{{\boldsymbol{#1}}}
\newcommand{\vz}{\bs{z}}
\newcommand{\va}{\bs{\alpha}}
\newcommand{\vb}{\bs{\beta}}
\newcommand{\vg}{\bs{\gamma}}

% Functions

\newcommand{\im}{\mathop{\mathrm{Im}}\nolimits}
\newcommand{\re}{\mathop{\mathrm{Re}}\nolimits}
\newcommand{\etp}{{\mathbf{e}}}

\newcommand{\PP}{{\mathrm P}_{\CC}}

\newcommand{\hol}[1]{\mathop{\mathcal{O}}\nolimits (#1)}
\newcommand{\mer}[1]{\mathop{\mathcal{M}}\nolimits (#1)}
\newcommand{\hog}[1]{\mathop{\mathcal{O}}_{#1}}

% ********************

\title{On holomorphicity of Hartogs series \\
satisfying algebraic relations}
\author{%
Hiroki Aoki%
\thanks{aoki\_hiroki\_math@nifty.com,
Faculty of Science and Technology,
Tokyo University of Science.}
\and
Kyoji Saito%
\thanks{saito@kurims.kyoto-u.ac.jp,
Research Institute for Mathematical Sciences,
Kyoto University.}
}

\maketitle

\begin{abstract}
We consider a formal power series in one variable
whose coefficients are holomorphic functions
in a given multidimensional complex domain.
Assume the following two conditions on the series.\\
{\bf{(C1)}}
The restriction of the series at each point
of a dense subset of the domain
converges in an open disk of a fixed radius.\\
{\bf{(C2)}}
The series is algebraic over the ring of holomorphic functions
on the direct product space of the domain and the disk.\\
The main theorem of the present note is that
the series defines a holomorphic function on the direct product space.
We also give an example where
the condition {\bf{(C2)}} is essentially necessary.
\end{abstract}

\section{Introduction and Main Theorem}

In the present note,
for a multidimensional complex domain \( U \) and \( \bs{u} \in U \), \(
\hol{U} \) denotes
the ring of holomorphic functions on \( U \) and \(
\hog{\bs{u}} \) denotes
the ring of germs of holomorphic functions at \( \bs{u} \).
We write the open disk of a radius \( r \) centered at the origin
as \( D(r):= \{ \; w \in \CC \; | \; |w|<r \; \} \).
Let \( D_1 \subset \CC^N \) be a domain.
Let \( R \) be a positive real number or \( + \infty \) and \( D_2 := D(R) \).
We write \( D:= D_1 \times D_2 \).

The objective of the present note is a formal power series,
which we call Hartogs series,
\begin{equation}
\label{eq:hs}
 F( \vz ,w):= \sum_{n=0}^{\infty} f_n ( \vz ) w^n \in \hol{D_1}[[w]]
 ,
\end{equation}
where \( f_n \in \hol{D_1} \).
In the present note,
we regard \( \hol{D} \) as a subring of \( \hol{D_1} [[w]] \) by
the power series expansion with respect to \( w \).
Thus, we have a question: ``Find a good condition for \( F \)
given by {\rm{({\ref{eq:hs}})}} to belong to \( \hol{D} \).''
It is known that
``If \( F|_{ \vz = \va } \in \hol{D_2} \) for all \( \va \in D_1 \),
then there exists a nowhere dense closed set \( A \in D_1 \) such
that \( F \in \hol{(D_1 \setminus A) \times D_2 } \)''
is true (\cite{TT1}*{Proposition 1}), however,
it may occur the case \( F \not\in \hol{D} \) even under this condition.
Several additional conditions
which guarantee that \( F \in \hol{D} \) were given
(cf.~\cite{TT1}).

In the present note, we consider this question
from the view point of algebraicity of \( F \) over \( \hol{D} \).
Our motivation is from the theory of automorphic forms.
Some kinds of automorphic forms of several variables have
Fourier-Jacobi series expansion.
Then, the converse: ``Whether a given formal Fourier-Jacobi series
with formal automorphicity converges or not?'' is
an interesting problem.
There are many researches about this.
First, in 2000, the above converse problem
was shown to be true by the first author in the case of
Siegel modular forms of degree \( 2 \) with
respect to the full modular group
(cf.~\cite{Ao1}).
Since then,
this problem has been shown to be true in various cases.
For example, Ibukiyama, Poor and Yuen extended
the above author's result to Siegel paramodular forms
of degree \( 2 \) with level \( \leqq 4 \) (cf.~\cite{IPY}).
Bruinier and Raum proved a modularity of symmetric formal
Fourier-Jacobi series in arbitrary degree
under a suitable assumptions (cf.~\cite{BR1}).
Very recently, the first author, Ibukiyama and Poor
proved a modularity of symmetric formal
Fourier-Jacobi series for Siegel paramodular forms
of degree \( 2 \) with arbitrary level (cf.~\cite{AIP}).
This result implies that
the algebraicity is important to show the convergence.
The aim of the present note is to prepare a general theorem
which can be applied to the above problem.

\bigskip

Here we consider the following two conditions:
\begin{description}
\item[(C1)]
There exists a dense subset \( S \) of \( D_1 \) such
that \( F |_{ \vz = \va } \) converges
and \( F |_{ \vz = \va } \in \hol{D_2} \) for all \( \va \in S \).
\item[(C2)]
\( F \) is algebraic over \( \hol{D} \).
Namely, there exists a nontrivial polynomial
\begin{equation}
\label{eq:af}
 \Phi ( \vz ,w,X):= \sum_{j=0}^t \Phi_j ( \vz ,w) X^{t-j} \in \hol{D} [X]
 \quad \bigl( \; t \geqq 1, \; \Phi_0 \not\equiv 0 \; \bigr)
\end{equation}
such that \( \Phi ( \vz ,w,F( \vz ,w)) \equiv 0 \) holds as an equality
on \( \hol{D_1} [[w]] \).

\end{description}
The main theorem of the present note is as follows:

\begin{theo}
\label{theo:main}
Let \( F \) be
a Hartogs series given by {\rm{(\ref{eq:hs})}} and
the conditions {\bf{(C1)}} and {\bf{(C2)}} hold.
Then \( F \) converges in \( D \) and \( F \in \hol{D} \).
\end{theo}

\begin{rema}
\label{rema:main}
The condition {\bf{(C2)}} is essentially necessary
in {\bf{Theorem} \ref{theo:main}}.
\end{rema}

\begin{rema}
The condition {\bf{(C2)}} alone already gives
some weaker holomorphicity result.
For every point \( ( \vz_0 ,0) \in D \),
the formal series defines an element \( F_{\vz_0} \) of
the completion \( \hat{A} = \CC [[ \vz - \vz_0 ,w]] \) of
the local ring \( A = \CC \{ \vz - \vz_0 ,w \} \) of \(
D \) at \( ( \vz_0 ,0) \).
As a ring of convergent power series with complex coefficients,
the ring \( A \) is a local integral domain
which is henselian and excellent.
By \cite{BR1}*{Proposition 4.2},
which is essentially a result from commutative algebra,
the ring \( A \) is algebraically closed in \( \hat{A} \).
By the condition {\bf{(C2)}} of the present note,
the formal series \( F_{\vz_0} \) is algebraic over \( A \).
This implies that \( F_{\vz_0} \) is already contained in \( A \) and
therefore convergent and holomorphic
in a neighborhood of \( ( \vz_0 ,0) \).
Therefore \( F \) is holomorphic
in a neighborhood of \( D_1 \times \{ 0 \} \).
%Therefore one gets the following statement:
%For every \( \vz_0 \in D_1 \), there exists a \( r( \vz_0 )>0 \) such
%that \( F( \vz ,w) \in \hol{ D( \vz_0 ; r( \vz_0 )) \times D(r( \vz_0 ))}
%\), where \( D( \vz_0 ; r( \vz_0 ) ) \) denotes
%the polydisc of radius \( r( \vz_0 ) \) around \( \vz_0 \).
%In other words, \( F \) is holomorphic
%in a neighborhood of \( D1 \times \{ 0 \} \).
>From this point of view,
the condition {\bf{(C1)}} is needed to be able to
holomorphically continue \( F \) to all of \( \hol{D} \).
\end{rema}

We give a proof of {\bf{Theorem} \ref{theo:main}} in the next section
and then we give an explicit example
which provides {\bf{Remark} \ref{rema:main}} in the last section.
Regarding the basics of function theory of several complex variables,
we referred to some standard textbooks (e.g. \cite{Nis}).

\section{Proof of the main theorem}

In this section,
we show {\bf{Theorem} \ref{theo:main}} by four steps.
Let \( p_1 : D \to D_1 \) be the projection \( ( \vz ,w) \mapsto \vz \).

\bigskip

\noindent
\underline{\bf{Step 1.}}

\bigskip

First, we show the following lemma.
\begin{lemm}
\label{lemm:rmpt}
Let \( F \) be a Hartogs series given by {\rm{({\ref{eq:hs}})}}.
Let \( U \subset D_1 \) be a domain and
let \( E \) be a nontrivial analytic subset of \( U \).
If \( F \in \hol{(U \setminus E ) \times D_2} \),
then \( F \in \hol{U \times D_2} \).
\end{lemm}
\begin{proof}
It is enough to show the case that \( E \) is
an analytic set of \( U \) with codimension \( 1 \).
Let \( \va = ( \alpha_s )_{s=1}^N \in E \) and
we take a neighborhood \( U_{\va} \) of \( \va \) in \( U \) such
that \( E \) is an analytic hypersurface in \( U_{\va} \).
By a suitable linear coordinate transformation if necessary,
we can write
\[
 E \cap U_{\va} =
 \{ \; \vz =( z_s )_{s=1}^N \in U_{\va} \; | \; g( \vz )=0 \; \}
\]
where \( g \in \hol{ U_{\va}} \) satisfies
the conditions \( g( \va )=0 \) and \( g( z_1, \alpha_2 , \dots , \alpha_N)
\not\equiv 0 \).
We take \( \varepsilon >0 \) and \(
\varepsilon_1 > \varepsilon_2 >0 \) such that \(
\overline{ U_1 } \subset U_{\va} \) and \(
E \cap \overline{ U_1 } = E \cap U_2^* \), where
\begin{align*}
 U_1
 & :=
 \{ \; \vz =( z_s )_{s=1}^N \in D_1 \; | \;
 | z_1 - \alpha_1 |< \varepsilon_1 , \;
 | z_j - \alpha_j |< \varepsilon \; (j=2, \dots ,N) \; \}
 \intertext{and}
 U_2^*
 & :=
 \{ \; \vz =( z_s )_{s=1}^N \in D_1 \; | \;
 | z_1 - \alpha_1 |< \varepsilon_2 , \;
 | z_j - \alpha_j | \leqq \varepsilon \; (j=2, \dots ,N) \; \}
 .
\end{align*}
Here we remark that \( \overline{ U_1 } \setminus U_2^* \) is a compact set
and \( E \cap ( \overline{ U_1 } \setminus U_2^*)= \emptyset \).
For any \( 0<r<R \),
let \( M_r := \max \{ \; |F( \vz ,w)| \; | \;
\vz \in \overline{ U_1 } \setminus U_2^* , \; w \in \overline{D(r)} \; \} \).
Then, by using Cauchy's inequality,
we have \( |f_n ( \vz )| \leqq M_r / r^n \) for
any \( \vz \in \overline{ U_1 } \setminus U_2^* \).
Thus, by using maximum principle, \( |f_n ( \vz )| \leqq M_r / r^n \) holds
for any \( \vz \in \overline{ U_1 } \).
Therefore, by Cauchy-Hadamard theorem,
we have \( F \in \hol{ U_1 \times D(r) } \) and
thus we have \( F \in \hol{U \times D_2} \).
\end{proof}

\bigskip

\noindent
\underline{\bf{Step 2.}}

\bigskip

Now,
let \( F \) be a Hartogs series given by {\rm{(\ref{eq:hs})}} and
the conditions {\bf{(C1)}} and {\bf{(C2)}} hold.
Let \( F^* ( \vz ,w) := \Phi_0 ( \vz ,w) F( \vz,w)
\in \hol{D_1} [[w]] \) and
\[
 \Phi^* ( \vz ,w,X):= X^t + \sum_{j=1}^t \Phi_j^* ( \vz ,w) X^{t-j}
 \in \hol{D} [X]
 ,
\]
where \( \Phi_j^* ( \vz ,w) = \Phi_j ( \vz ,w) \Phi_0 ( \vz ,w)^{j-1}
\in \hol{D} \).
Since \( \Phi ( \vz ,w,F( \vz ,w)) \equiv 0 \),
we have \( \Phi^* ( \vz ,w,F^* ( \vz ,w)) \equiv 0 \).
\begin{lemm}
\label{lemm:mkmc}
Let \( F \) and \( F^* \) be as above.
If \( F^* \in \hol{D} \), we have \( F \in \hol{D} \).
\end{lemm}
\begin{proof}
Since \( F^* \in \hol{D} \),
we have \( F \equiv F^* / \Phi_0 \in \mer{D} \).
Let \( E_{ \Phi_0 } :=
\{ \; \va \in D_1 \; | \; \Phi_0 |_{ \vz= \va } \equiv 0 \} \).
We remark that \( E_{ \Phi_0 } \) is
a nontrivial analytic subset of \( D_1 \).
First we show that \( F \in \hol{D^{\prime}} \),
where \( D^{\prime} := D \setminus p_1^{-1} ( E_{ \Phi_0 } ) \).
Let \( ( \vb , c) \in D^{\prime} \) be
a singular point of \( F \),
namely, \( ( \vb , c) \) is a pole or an indeterminacy.
We take \( V \subset D^{\prime} \),
which is an open neighborhood of \( ( \vb , c) \) such
that \( F \equiv P/Q , \; P,Q \in \hol{V} \) and \( P \) and \( Q \) are
coprime in any \( \hog{( \vz ,w)} \; (( \vz ,w) \in V) \).
Let \( T_Q := \{ \; ( \vz ,w) \in V \; | \; Q( \vz ,w)=0 \; \} \) and \(
T^* := \{ \; ( \vz ,w) \in T_Q \; | \; P( \vz ,w)=0 \; \} \).
Namely, \( T_Q \) and \( T^* \) are the polar set and
the indeterminacy set of \( F \) in \( V \), respectively.
By general theory, \( T_Q \) and \( T^* \) are
analytic subsets of \( V \) with
codimension \( 1 \) and \( 2 \), respectively.
Because the map \( p_1 |_{T_Q} \) is locally finite,
it is an open map and
hence \( p_1 ( T_Q \setminus T^* ) \) is a non-empty
open subset of \( D_1 \).
Thus there exists \( \vg \in S \cap p_1 ( T_Q \setminus T^* ) \).
On the equality \( QF \equiv P \) on \( \hol{D_1} [[w]] \),
there exists \( d \in D_2 \) such that \( F \) does not
converge at \( ( \vg ,d) \in D \), because \( 
\vg \in p_1 ( T_Q \setminus T^* ) \).
It contradicts with \( \vg \in S \).
Therefore we have \( F \in \hol{D^{\prime}} \).
Consequently, we have \( F \in \hol{D} \) by {\bf{Lemma \ref{lemm:rmpt}}}.
\end{proof}
>From the above, when \( t=1 \), {\bf{Theorem} \ref{theo:main}} has
been proven.

\bigskip

\noindent
\underline{\bf{Step 3.}}

\bigskip

Continuing from the previous step,
let \( F \) be a Hartogs series given by {\rm{(\ref{eq:hs})}} and
the conditions {\bf{(C1)}} and {\bf{(C2)}} hold.
We may assume \( t \geqq 2 \).
In this step,
we assume \( \Phi_0 ( \vz ,w) \equiv 1 \).
Let \( \Delta ( \vz ,w) \in \hol{D} \) be
the discriminant of \( \Phi \in \hol{D} [X] \) and \(
T_{\Delta} := \{ \; ( \vz ,w) \in D \; | \; \Delta ( \vz ,w)=0 \; \} \).
\begin{lemm}
\label{lemm:acnt}
Let \( F, \Phi , \Delta \) and \( T_{\Delta} \) be as above.
If \( D_1 \) is simply connected and
if there exist \( f_1, f_2 , \dots , f_m \in \hol{D_1} \) such
that \( T_{\Delta} = \bigsqcup_{k=1}^m \{ \; ( \vz ,w) \in D \; | \;
w= f_k ( \vz ) \; \} \) (disjoint union), then \( F \in \hol{D} \).
\end{lemm}
\begin{proof}
We choose a point \( ( \vb ,c) \in (S \times D_2) \setminus T_{\Delta} \).
Since \( \Delta ( \vb ,c) \neq 0 \), we can choose
a unique branch \( G \in \hog{( \vb ,c)} \) of the algebraic function
defined by \( \Phi \) which
takes the same value as \( F \) at \( ( \vb ,c) \).
By the assumption of this lemma,
the fundamental group \( \pi_1 ( D \setminus T_{\Delta} , ( \vb ,c) ) \)
is isomorphic to \( \pi_1 ( D_2 \setminus T_{\Delta} |_{ \vz = \vb } ,c) \).
Since \( D \setminus T_{\Delta} \) is a pathwise connected set and
since \( F |_{ \vz = \vb } (w) \) is a single valued holomorphic function,
we have a unique holomorphic
function \( \widetilde{G} \in \hol{D \setminus T_{\Delta}} \) which
is an analytic continuation of \( G \) along
all paths in \( D \setminus T_{\Delta} \).
Here \( \Phi ( \vz ,w, \widetilde{G} ( \vz ,w)) \equiv 0 \) holds as
an equality in \( \hol{D \setminus T_{\Delta}} \).
Because \( \Phi \in \hol{D} [X] \), this \( \widetilde{G} \) can
be extended to \( D \) as a continuous function.
Hence, by Rad{\'o}'s theorem, we can recognize \( \widetilde{G}
\in \hol D \).
In consequence, we have \( \widetilde{G} \in \hol D \) satisfying \(
\Phi ( \vz ,w, \widetilde{G} ( \vz ,w)) \equiv 0 \) as
an equality in \( \hol{D} \) and \( \widetilde{G} |_{ \vz = \vb } \equiv
F |_{ \vz= \vb } \in \hol{D_2} \).

Because \( \Phi ( \vz ,w, \widetilde{G} ( \vz ,w)) \equiv 0 \), \(
\Phi \in \hol{D} [X] \) and \( \widetilde{G} \in \hol{D} \),
we know that \( X- \widetilde{G} \) divides \( \Phi \).
Namely, we have \( \Phi ( \vz ,w,X)=(X- \widetilde{G} ) \Psi ( \vz ,w,X) \),
where \( \Psi ( \vz ,w,X) \in \hol{D}[X] \) and \( \deg_X \Psi =t-1 \).
This equality also holds as an equality in \( \hol{ D_1 }[[w]][X]
\) and \( \Phi ( \vz ,w,F( \vz ,w)) \equiv 0 \) holds in this sense.
Therefore, \( F \equiv G \) or \( \Psi ( \vz ,w,F( \vz ,w)) \equiv 0
\) holds as an equality on \( \hol{ D_1 }[[w]] \).
In the former case,
we have \( F \equiv G \in \hol{D} \).
In the later case, by using induction of \( t \),
we have \( F \in \hol{D} \).
\end{proof}

\bigskip

\noindent
\underline{\bf{Step 4.}}

\bigskip

In this step, we finish the proof.
Let \( F \) be a Hartogs series given by {\rm{(\ref{eq:hs})}} and
the conditions {\bf{(C1)}} and {\bf{(C2)}} hold.
By the arguments of {\bf{Step 1}} and {\bf{Step 2}},
we may assume \( t \geqq 2 \) and \( \Phi_0 ( \vz ,w) \equiv 1 \).
Let \( \Delta ( \vz ,w) \in \hol{D} \) be
the discriminant of \( \Phi \in \hol{D} [X] \) and
let \( E_{\Delta} :=
\{ \; \vb \in D_1 \; | \; \Delta |_{ \vz = \vb } \equiv 0 \} \).
We remark that \( E_{\Delta} \) is
a nontrivial analytic subset of \( D_1 \).

First we show that \( F \in \hol{D^{\prime}} \),
where \( D^{\prime} := D \setminus p_1^{-1} ( E_{\Delta} ) \).
Let \( \vb \in D_1 \setminus E_{\Delta} \).
For any \( c \in D_2 \),
we take two simply connected open neighborhoods \(
U_c \) of \( \vb \) in \( D_1 \setminus E_{\Delta} \) and \(
W_c \) of \( c \) in \( D_2 \) as follows:
\begin{itemize}
\item
If \( \Delta ( \vb ,c) \neq 0 \),
we take \( U_c \) and \( W_c \) such that \(
\Delta ( \vz ,w) \) does not vanish in \( U_c \times W_c \).
\item
If \( \Delta ( \vb ,c)=0 \),
we take not only \( U_c \) and \( W_c \) but
also \( P_c \in \hol{ U_c } [w] \) such that
the following four conditions hold.
\begin{itemize}
\item
There exists \( m \in \NN \) such that \( P_c |_{ \vz = \vb } = (w-c)^m \).
Namely, \( P_c \) is a distinguished pseudopolynomial
(Weierstrass polynomial) with its center \( ( \vb ,c) \).
\item
In \( U_c \times W_c \), \( \Delta ( \vz ,w)=0 \) if
and only if \( P_c ( \vz ,w)=0 \).
\item
The discriminant of \( P_c \) is not identically zero in \( U_c \).
\item
For any \( \vg \in U_c \), the number of
the zeros of \( P_c |_{\vz = \vg} \in \CC [w] \) in \( W_c \) is
exactly \( m \), counted with multiplicity.
\end{itemize}
\end{itemize}
Then, for any \( 0<r<R \), we have
\[
 \{ \vb \} \times \overline{D(r)}
 \subset
 \bigcup_{ c \in \overline{D(r)}}
 \left( U_c \times W_c \right)
 .
\]
Because \( \{ \vb \} \times \overline{D(r)} \) is compact,
we can take a simply connected open neighborhood \( U \) of \( \vb
\) in \( D_1 \setminus E_{\Delta} \) and \( \varepsilon >0 \) such that
\begin{align*}
 & \quad \;
 \{ \; ( \vz ,w) \in U \times D(r+ \varepsilon )
 \; | \;
 \Delta ( \vz ,w)=0 \; \}
 \\
 & =
 \{ \; ( \vz ,w) \in U \times D(r+ \varepsilon )
 \; | \;
 P( \vz ,w)=0 \; \}
 ,
\end{align*}
where
\[
 P( \vz ,w) := \prod_{|c| \leqq r , \; \Delta ( \vb ,c)=0}
 P_c ( \vz ,w) \in \hol{U}[w]
 .
\]
Let \( \delta \in \hol{U} \) be the discriminant of \( P \) and
let \( E_{\delta} := \{ \; z \in U \; | \; \delta (z)=0 \; \} \).
We remark that \( E_{\delta} \) is
an analytic set of \( U \) with codimension \( 1 \).
For any \( \vg \in U \setminus E_{\delta} \),
there exists a simply connected open
neighborhood \( U_{\vg} \) of \( \vg \) in \( U \setminus E_{\delta} \).
Here in \( U_{\vg} \times D(r+ \varepsilon ) \), \( P( \vz ,w) \) is
decomposed into a product of linear factors in \( \hol{U_{\vg}}[w] \) and
zeros of each factor does not intersect.
Thus, by {\bf{Lemma \ref{lemm:acnt}}},
we have \( F \in \hol{U_{\vg} \times D(r+ \varepsilon )} \).
Hence we have \( F \in \hol{(U \setminus E_{\delta} )
\times D(r+ \varepsilon )} \).
By {\bf{Lemma \ref{lemm:rmpt}}}, we have \( F \in
\hol{ U \times D(r+ \varepsilon )} \),
therefore we have \( F \in
\hol{( D_1 \setminus E_{\Delta} ) \times D(r+ \varepsilon ) } \) and
consequently we have \( F \in \hol{D^{\prime}} \).

Again by {\bf{Lemma \ref{lemm:rmpt}}}, we have \( F \in \hol{D} \) and
now we complete the proof of {\bf{Theorem} \ref{theo:main}}.

\begin{rema}
In the present note we assume that \( D_1 \) is a domain in \( \CC^N \).
However, {\bf{Theorem} \ref{theo:main}} holds for any
open complex manifold with almost the same proof.
\end{rema}

\section{Example}

In this section, we show that
the condition {\bf{(C2)}} is essentially necessary
in {\bf{Theorem} \ref{theo:main}}.
Namely, we give a Hartogs series \( F \) which
satisfies {\bf{(C1)}} but does not converge
in \( D \setminus S \).
For simplicity, we only treat the case \( N=1 \).

Let \( \HH := \{ \; \tau \in \CC \; | \; \im \tau >0 \; \} \)
be the upper half complex plane.
It is well known that the elliptic theta function
\[
 \theta_{11} (z; \tau ):= \sum_{m \in \ZZ}
 \exp \left( \pi \sqrt{-1} \left(
 \left( m + \tfrac{1}{2} \right)^2 \tau
 +2 \left( m + \tfrac{1}{2} \right) \left( z+ \tfrac{1}{2} \right)
 \right) \right)
 \in \hol{\CC \times \HH}
\]
has the following properties:
\begin{itemize}
\item
\( \theta_{11} (z; \tau )=0 \) if and only if \( z \in \ZZ \tau + \ZZ \).
\item
\( \theta_{11} (z; \tau )=
\exp \left( \pi \sqrt{-1} \left( a^2 \tau +2az \right) \right)
\theta_{11} (z+a \tau +b ; \tau ) \) for any \( a,b \in \ZZ \).
\item
For any \( \tau_0 \in \HH \), \( \theta_{11} |_{\tau = \tau_0}
\) has a zero of order \( 1 \) at the origin \( z=0 \).
\end{itemize}
Now we fix a point \( \tau_0 \in \HH \) and
consider the case
\[
 f_n (z) =
 \left\{
 \begin{array}{ll}
 n^n \theta_{11} ( n! z; \tau_0 ) & (n \geqq 1)
 \\[4pt]
 0 & (n=0)
 \end{array}
 \right.
 .
\]
Then, as for
\[
 F(z,w):= \sum_{n=0}^{\infty} f_n (z) w^n \in \hol{\CC}[[w]]
 ,
\]
the following two properties hold.
\begin{itemize}
\item
\( F|_{z= \alpha } (w) \) converges and is holomorphic on \( \CC \)
if \( \alpha \in \QQ \tau_0 + \QQ \).
\item
\( F|_{z= \alpha } (w) \) does not converge
if \( \alpha \not\in \QQ \tau_0 + \QQ \).
\end{itemize}
The former one is clear, actually, \(
F|_{z= \alpha } (w) \in \hol{\CC}[w] \).
To show the latter one, we need to prepare some lemmas.
For any \( x \in \RR \), we define
\[
 e(x):= \left\{ \begin{array}{ll}
 [x] & ( \; x-[x] \leqq 1/2 \; ) \\[4pt]
 {} [x]+1 & ( \; x-[x]>1/2 \; )
 \end{array} \right.
 \quad \text{and} \quad
 d(x):=x-e(x)
 .
\]
Namely, \( e(x) \) is the nearest integer of \( x \) and \(
d(x) \) is the signed minimal distance from \( x \) to \( \ZZ \).
We remark that \( |d(x)| \leqq 1/2 \).
\begin{lemm}
\label{lemm:the1}
If \( x \in \RR \setminus \QQ \),
there exist infinitely many \( n \)'s
such that \( |d(n!x)| \geqq \frac{1}{2(n+1)} \) holds.
\end{lemm}
\begin{proof}
We define \( \{ a_n \}_{n=1}^{\infty} \) and \(
\{ b_n \}_{n=1}^{\infty} \) by
\[
 a_1 := e(x) , \quad b_1 := d(x) , \quad
 a_k := e(k b_{k-1}) , \quad b_k :=d(k b_{k-1})
 ,
\]
inductively.
Then we have a multiple factoradic representation
\[
 x= \left( \sum_{k=1}^K \frac{a_k}{k!} \right) + \frac{b_K}{K!}
 = \sum_{k=1}^{\infty} \frac{a_k}{k!}
 , \quad
\]
and there exist infinitely many \( n \)'s
such that \( a_n \neq 0 \), since \( x \) is irrational.
If \( a_n \neq 0 \; (n \geqq 2) \), we have
\[
 |d((n-1)!x)| = \left| b_{n-1} \right|
 = \left| \frac{a_n}{n} + \frac{b_n}{n} \right|
 \geqq \left| \frac{a_n}{n} \right| - \left| \frac{b_n}{n} \right|
 \geqq \frac{1}{n} - \frac{1}{2n} = \frac{1}{2n}
 .
\]
\end{proof}
We define \( x,y : \CC \to \RR \) by \( z=x(z) \tau_0 + y(z) \).
Let \( e(z):=e(x(z)) \tau_0 + e(y(z)) \) and \( d(z):=z-e(z) \).
\begin{lemm}
\label{lemm:the2}
For any \( n \in \NN \) and \( z \in \CC \), \(
\left| \theta_{11} (n!z; \tau_0) \right| \geqq
\left| \theta_{11} (d(n!z); \tau_0)) \right| \) holds.
\end{lemm}
\begin{proof}
Since \( n!z=d(n!z)+e(n!z) \) and \( e(n!z) \in \ZZ \tau_0 + \ZZ \),
we have
\[
 \theta_{11} (d(n!z); \tau_0 )=
 \exp \left( \pi \sqrt{-1} \left(
 e(x(n!z))^2 \tau_0 +2e(x(n!z))d(n!z) \right) \right)
 \theta_{11} (n!z ; \tau_0 )
\]
and thus we have
\begin{align*}
 & \left| \theta_{11} (n!z ; \tau_0 ) \right|
 \\
 = &
 \exp \left( \pi \left(
 e(x(n!z))^2 \im \tau_0 +2e(x(n!z)) \im d(n!z) \right) \right)
 \left| \theta_{11} (d(n!z); \tau_0 ) \right|
 \\
 = &
 \exp \left( \pi \left(
 e(x(n!z))^2 +2e(x(n!z)) d(n!x(z)) \right) \im \tau_0 \right)
 \left| \theta_{11} (d(n!z); \tau_0 ) \right|
 \\
 \geqq &
 \exp \left( \pi \left(
 e(x(n!z))^2 -|e(x(n!z))| \right) \im \tau_0 \right)
 \left| \theta_{11} (d(n!z); \tau_0 ) \right|
 \\
 \geqq &
 \left| \theta_{11} (d(n!z); \tau_0 ) \right|
 .
\end{align*}
\end{proof}
\begin{lemm}
\label{lemm:the3}
There exists a positive constant \( c>0 \) such
that \( \left| \theta_{11} (z; \tau_0) \right| \geqq c|z| \) holds
for any \( z \in \Lambda := \{ \; z \in \CC \; | \; |x(z)| \leqq 1/2 , \;
|y(z)| \leqq 1/2 \; \} \).
\end{lemm}
\begin{proof}
Let \( g(z):= \theta_{11} (z; \tau_0)/z \).
Obviously, \( g \) is a meromorphic function on \( \CC \) with
a singularity at the origin.
Because \( \theta_{11} |_{\tau = \tau_0}
\) has a zero of order \( 1 \) at the origin,
this singularity is removable.
Hence by setting \( g(0):= \lim_{z \to 0} g(z) \neq 0 \),
we can regard \( g \) as a holomorphic function on \( \CC \) and
the set of all zeros of \( g \) is \( ( \ZZ \tau_0 + \ZZ ) \setminus
\{ 0 \} \). 
Therefore,
by taking \( c:= \min \{ \; |g(z)| \; | \; z \in \Lambda \; \} >0 \),
we have the assertion.
\end{proof}
\begin{prop}
\( F|_{z= z_0} (w) \) does not converge
if \( z_0 \not\in \QQ \tau_0 + \QQ \).
\end{prop}
\begin{proof}
Let \( M:=( \im \tau_0 ) \min \{ 1, | \tau_0 |^{-1} \} \).
Since \( |z| \geqq | \im z |=( \im \tau_0 )|x(z)| \) and
since \( |z|=| \tau_0 || \tau_0^{-1} z| \geqq | \tau_0 || \im ( \tau_0^{-1} z)|
=| \tau_0 || \im \tau_0^{-1} |
|y(z)|= \frac{ \im \tau_0 }{| \tau_0 |} |y(z)| \),
we have \( |z| \geqq M \max \{ |x(z)| , |y(z)| \} \).
We assume \( \alpha \not\in \QQ \tau_0 + \QQ \).
By {\bf{Lemma \ref{lemm:the2}}}, {\bf{Lemma \ref{lemm:the3}}} and
{\bf{Lemma \ref{lemm:the1}}}, there exist infinitely many \( n \)'s
such that \( \left| \theta_{11} (n! \alpha ; \tau_0) \right| \geqq
\frac{cM}{2(n+1)} \) holds.
Hence we know that \( F|_{z= \alpha } (w) \) does not converge by
Cauchy-Hadamard theorem.
\end{proof}

\section*{Acknowledgement}

The authors would like to thank
Professor Kazuko Matsumoto
and
Professor Cris Poor
for their useful advice.
We also thank anonymous referee for a careful reading
of the manuscript and for many helpful comment.
This work is supported
by JSPS KAKENHI Grant Numbers 23K03039 and 23K25765.

\end{document}